\documentclass{article}
\usepackage[dvipsnames]{xcolor}
\usepackage{gastex}
\usepackage{amsmath}
\usepackage{lineno}
    \linenumbers
\usepackage{url}
\usepackage{theorem}
\newtheorem{theorem}{Theorem}
\newtheorem{lemma}[theorem]{Lemma}
\newtheorem{proposition}[theorem]{Proposition}
\newtheorem{corollary}[theorem]{Corollary}
{\theorembodyfont{\rmfamily}%
  \newtheorem{example}[theorem]{Example}
   }

\newenvironment{proof}{\noindent\textit{Proof.}}
{\QED\vskip\theorempostskipamount} 
\newenvironment{proofof}[1]{\noindent\textit{Proof
    \protect{#1}.}}
                       {\QED\vskip\theorempostskipamount}
\def\petitcarre{\vrule height4pt width 4pt depth0pt}
\def\QED{\relax\ifmmode\eqno{\hbox{\petitcarre}}\else{%
  \unskip\nobreak\hfil\penalty50\hskip2em\hbox{}\nobreak\hfil
  \petitcarre
  \parfillskip=0pt \finalhyphendemerits=0\par\smallskip}
  \fi}
\DeclareMathOperator{\Card}{Card}
\begin{document}
\title{Note on sets of first return}
\author{Val\'erie Berth\'e$^1$, Clelia De Felice$^3$, 
Francesco Dolce$^2$, Dominique Perrin$^2$,
  \\
Christophe  Reutenauer$^4$,
Giuseppina Rindone$^2$\\\\
$^1$CNRS, Universit\'e Paris 7,
$^2$Universit\'e Paris Est,\\ 
$^3$Universit\`a degli Studi di Salerno, $^4$Universit\'e du Qu\'ebec \`a Montr\'eal}
\maketitle

\tableofcontents
\begin{abstract}
We prove that for an alphabet $A$ with three letters,
 the set of first return to a given word
in a uniformly recurrent set satisfying the tree condition is a basis of the free group
on $A$.
\end{abstract}
\section{Introduction}
Tree sets are symbolic dynamical systems subject to a restriction
on the possible extensions of a given word. These sets
are introduced
in~\cite{BertheDeFeliceDolcePerrinReutenauerRindone2013}
as a common generalization of Sturmian sets and of regular
interval exchange sets.

We prove in this note that, for an alphabet $A$ with three letters,
 the set of first return to a given word
in a uniformly recurrent set satisfying the tree condition is a basis of the free group
on $A$.

In the first section, we prove some preliminary results concerning
bispecial words. We define extension graphs and tree sets.
In Section~\ref{sectionMainResult} we state our main result
(Theorem~\ref{proposition1}). The proof uses a classification
of possible Rauzy graphs. We end we some examples.
\section{Preliminaries}\label{sectionPreliminaries}
In this section, we first recall some definitions concerning words.
We give the definition of recurrent and uniformly recurrent
sets of words. (see~\cite{BerstelDeFelicePerrinReutenauerRindone2012} for a more detailed
presentation). We also give the definitions and basic
properties of tree sets. We prove some prelimary properties
of sets of first return words.
\subsection{Recurrent sets}
Let $A$ be a finite nonempty alphabet. All words considered below,
unless
stated explicitly, are supposed to be on the alphabet $A$.
We denote by $A^*$ the set of all words on $A$.
We denote by $1$ or by $\varepsilon$ the empty word.

For a set $X$ of words and a word $u$, we denote
\begin{displaymath}
u^{-1}X=\{v\in A^*\mid uv\in X\}
\end{displaymath}
the right \emph{residual} of $X$ with respect to $u$.
This notation should not be confused with the notation
for the inverse in the free group on $A$.

A set of words is said to be \emph{factorial} if it contains the
factors of its elements.

Let $F$ be a set of words on the alphabet $A$.
For a word $w\in F$, we denote
\begin{eqnarray*}
L(w)&=&\{a\in A\mid aw\in F\}\\
R(w)&=&\{a\in A\mid wa\in F\}\\
E(w)&=&\{(a,b)\in A\times A\mid awb\in F\}
\end{eqnarray*}
and further
\begin{displaymath}
\ell(w)=\Card(L(w)),\quad r(w)=\Card(R(w)),\quad e(w)=\Card(E(w)).
\end{displaymath}
A word $w$ is \emph{right-extendable} if $r(w)>0$,
\emph{left-extendable} if $\ell(w)>0$ and \emph{biextendable} 
if $e(w)>0$. A factorial set
$F$ is called \emph{right-essential}
(resp. \emph{left-essential}, resp. \emph{biessential}) if every word in $F$ is
right-extendable (resp. left-extendable, resp. biextendable).

A word is \emph{left-special} if $\ell(w)\ge 2$, \emph{right-special} if
$r(w)\ge 2$, 
\emph{bispecial} if it is both left-special and right-special.

A set of words $F$ is \emph{recurrent} if it is factorial and if for every
$u,w\in F$ there is a $v\in F$ such that $uvw\in F$. A recurrent set
$F\ne\{1\}$ is biessential.

A set of words $F$ is said to be \emph{uniformly recurrent} if it is
right-essential and if, for any word $u\in F$, there exists an integer $n\ge
1$
such that $u$ is a factor of every word of $F$ of length $n$.
A uniformly recurrent set is recurrent.
\subsection{Tree sets}\label{sectionTreeSets}
Let $F$ be a biessential set.
For  $w\in F$, we
 consider the \emph{extension graph}
of $w$ as the undirected graph $G(w)$ on the set of vertices
which is the disjoint union of
$L(w)$ and $R(w)$ with edges the pairs 
$(a,b)\in E(w)$.

Recall that an undirected graph is a tree if it is connected and acyclic.

We say that $F$ is a \emph{tree set} if it is biessential
and if for every word
$w\in F$, the graph $G(w)$ is a tree. A tree set 
has complexity $kn+1$ with $k=\Card(A\cap F)-1$
(see~\cite{BertheDeFeliceDolcePerrinReutenauerRindone2013},
Proposition 3.2).

Let $F$ be a set of words. For $w\in F$, and $U,V\subset A^*$,
let 
$U(w)=\{\ell\in U\mid \ell w\in F\}$
 and let $V(w)=\{r\in V\mid wr\in F\}$.
The \emph{generalized extension graph} of $w$ relative to
$U,V$ is the following undirected graph $G_{U,V}(w)$. The set of vertices is
made of two disjoint copies of $U(w)$ and $V(w)$.
 The edges are the pairs $(\ell,r)$
for $\ell\in U(w)$ and $r\in V(w)$
such that $\ell wr\in F$. The extension graph $G(w)$ defined previously
corresponds
to the case where $U,V=A$.

\begin{example}
Let $F$ be the Fibonacci set. Let $w=a$, $U=\{aa,ba,b\}$ and let
$V=\{aa,ab,b\}$. The graph $G_{U,V}(w)$ is represented in
Figure~\ref{figureStrongTree}.
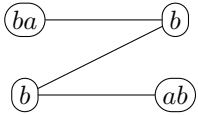
\begin{figure}[hbt]
\centering
\gasset{AHnb=0, Nadjust=wh}
\begin{picture}(20,10)
\node(b)(0,0){$b$}\node(ba)(0,10){$ba$}
\node(ab)(20,0){$ab$}\node(b')(20,10){$b$}

\drawedge(ba,b'){}\drawedge(b,b'){}\drawedge(b,ab){}
\end{picture}
\caption{The graph $G_{U,V}(w)$.}\label{figureStrongTree}
\end{figure}
\end{example}

The following property is proved in~\cite{BertheDeFeliceDolcePerrinReutenauerRindone2013d}.
It shows that in a tree set, not only
the extension graphs but all generalized extension graphs
are trees.
\begin{proposition}\label{PropStrongTreeCondition}
Let $F$ be a tree set.  For any $w\in F$, any
 finite 
$F$-maximal suffix code  $U\subset F$ 
and any  finite  $F$-maximal prefix code $V\subset F$,
 the generalized extension graph $G_{U,V}(w)$ is  a tree.
\end{proposition}
Note that the statement also holds under the (apparently) weaker
hypothesis that $U$ is a finite $Fw^{-1}$-maximal suffix code
and $V$ a finite $w^{-1}F$-maximal prefix code. Indeed, let then
$U'$ be a finite $F$-maximal suffix code containing $U$
and $V'$ be a finite $F$-maximal prefix code containing $V$.
Then $U'(w)=U(w)$ and $V'(w)=V(w)$. Thus $G_{U,V}(w)=G_{U',V'}(w)$
and thus $G_{U,V}(w)$ is a tree.
\subsection{Sets of first right return words}
For $x\in F$, let $\Gamma_F(x)=\{r\in F\mid xr\in F\cap A^+x\}$ 
be the set of \emph{right return words} to $x$ and let
$R_F(x)=\Gamma_F(x)\setminus\Gamma_F(x) A^+$ be the 
\emph{set of first right  return words} to $x$.
A word $x$ is \emph{unioccurrent} in a word $y$ if there exist unique
words $u,v$ such that $y=uxv$.

A set $F$ is \emph{periodic} if it is the set of factors of $w^*$
for some word $w$. If $\Card(A)\ge 2$, a set satisfying the
tree condition is not periodic since otherwise its complexity
would be bounded by a constant. 
\begin{proposition}\label{propositionBispecial}
Let $F$ be a recurrent set which is not periodic.
Any word $x\in F$ is a factor of a bispecial word $y$. There is a unique
shortest bispecial word $y$ containing $x$ and $x$ is unioccurrent
in $y$. If $y=uxv$, one has
$R_F(x)=vR_F(y)v^{-1}$.
\end{proposition}
Note that in the above statement, since all the elements of
$R_F(y)$ end with $v$, the notation $v^{-1}$ can be indifferently
interpreted as the residual or the inverse of $v$ in the free group.

Let $F$ be a factorial set.
The \emph{Rauzy graph} of order $n\ge 1$ is the following labeled
graph $G_n$. Its vertices are the words in the set $F\cap A^n$.
Its edges are the triples $(x,a,y)$ for all $x,y\in F\cap A^n$
and $a\in A$ such that $xa\in F\cap Ay$.

Observe that when $F$ is recurrent, any Rauzy graph
is strongly connected. Indeed, let $u,w\in F\cap A^n$.
Since $F$ is recurrent, there is a $v\in F$ such
that $uvw\in F$. Then there is a path
in $G_n$ from $u$ to $w$ labeled $wv$.

\begin{lemma}\label{lemma1}
Let $F$ be a recurrent set.
If $F$ is not periodic, any word of $F$ is a prefix of a right-special word
and a suffix of a left-special word. 
\end{lemma}
\begin{proof}
Let $x\in F$ be of length $n$. If $F$ is not periodic, there is a 
 vertex of the Rauzy graph $G_n$ which is the origin of more
than one edge (otherwise, the graph $G_n$ is reduced to
a cycle and $F$ is periodic).
Thus there is a word $y$ of length $n$ which is right-special. Thus
the label $v$ of path from $x$ to $y$ is  such that $xv$ is right-special.

The proof that $x$ is a suffix of a left-special word is symmetrical.
\end{proof}

\begin{lemma}\label{lemma2}
Let $F$ be a recurrent set.
If $F$ is not periodic, any word in $F$ is a factor of a bispecial word.
More precisely, if $u$ is of minimal length such that $ux$
is left-special and $v$ of minimal length such that $xv$
is right-special, then $y=uxv$ is a bispecial word containing $x$
of minimal length and $x$ is unioccurrent in $y$.
\end{lemma}
\begin{proof}
Let us show by induction on the length of a prefix $v'$ of $v$
that $auxv'\in F$ for any $a\in L(ux)$. This is true if $v'$
is empty. Otherwise, set $v'=v''b$ with $b\in A$. For any $a\in
L(ux)$,
we have $auxv''\in F$ by induction hypothesis. On the other hand,
by the minimality of $v$, $xv''$ is not right-special
and thus $auxv''$ is not right-special. This forces $auxv''b\in F$,
which proves the property for $v'$. This implies that $uxv$ is
left-special. The proof that $uxv$ is right-special is symmetrical.
 If $y=pxs$ is a bispecial word containing  $x$, then
$px$ is left-special and $xs$ is right-special.
Thus $|p|\ge |u|$ and $|s|\ge|v|$ and thus $|y|\ge|uxv|$.
Finally, assume that $x$ has a second occurrence in $y$. Then
$y=u'xv'$ and either $|u'|<|u|$ or $|v'|<|v|$. Both
options are impossible and thus $x$ is unioccurrent in $y$.
\end{proof}
\begin{proofof}{of Proposition \ref{propositionBispecial}}
We may assume $\Card(A)\ge 2$.
The first assertion is a consequence of Lemma~\ref{lemma2}.

For the
second assertion, let $u$ be of minimal length such that
$ux$ is left-special and $v$ of minimal length such that
$xv$ is left-special. By Lemma~\ref{lemma2}, $y=uxv$ is a bispecial
word containing $x$ of minimal length and $x$ is unioccurrent in $y$. 
Since $u$ is of minimal length, any extension of $x$ to
the left is comparable with $u$ for the suffix order.
In the same way any extension of $x$ to the right
is comparable with $v$ for the prefix order.
Thus there cannot be another
bispecial word $u'xv'$ of the same length since otherwise
$|u'|<|u|$, contradicting the hypothesis on $u$ since
$u'x$ is left-special or $|v'|<|v|$, contradicting
the hypothesis on $v$. Assume that $r\in R_F(x)$. Set $xr=sx$.
Then $u$ is comparable with $s$ for the prefix order
 Similarly $v$ is comparable
with $r$ for the prefix order. This implies that
$u$ is a suffix of $s$ and $v$ is a prefix of $r$ since
otherwise $x$ would have a second occurrence in $y$. 
\begin{figure}[hbt]
\gasset{Nadjust=wh,AHnb=0}\centering
\begin{picture}(80,20)
\node(1)(-5,20){}\node(2)(10,20){}\node(3)(30,20){}\node(4)(45,20){}
\node(1')(10,15){}\node(2')(30,15){}\node(3')(60,15){}
\node(1'')(10,10){}\node(2'')(40,10){}\node(3'')(60,10){}
\node(1''')(25,5){}\node(2''')(40,5){}\node(3''')(60,5){}\node(4''')(75,5){}

\drawedge(1,2){$u$}\drawedge(2,3){$x$}\drawedge(3,4){$v$}
\drawedge(1',2'){$x$}\drawedge(2',3'){$r$}
\drawedge(1'',2''){$s$}\drawedge(2'',3''){$x$}
\drawedge(1''',2'''){$u$}\drawedge(2''',3'''){$x$}\drawedge(3''',4'''){$v$}
\end{picture}
\caption{First returns to $x$ and $y$.}\label{figFirstReturns}
\end{figure}
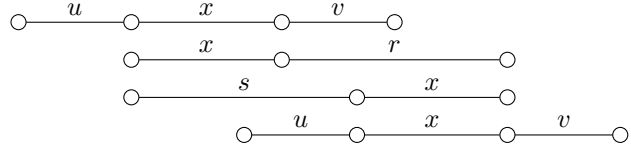

Set $s=s'u$ and $r=vr'$. Then
\begin{displaymath}
yv^{-1}rv=uxvv^{-1}rv=uxvr'v=uxrv=usvx=us'uxv=us'y.
\end{displaymath}
This
implies that $v^{-1}rv$ is in $\Gamma_F(y)$ (see
Figure~\ref{figFirstReturns}). Since $r\in R_F(x)$, this
implies acually that  $v^{-1}rv\in R_F(y)$
The converse holds for the same reasons.
\end{proofof}

\begin{example}
Let $A=\{a,b\}$ and let $F$ be the Fibonacci set. The shortest
bispecial word containing $x=aa$ is $y=abaaba$. One has
$R_F(x)=\{baa,babaa\}$ and $R_F(y)=\{aba,baaba\}$. Thus
$R_F(x)=baR_F(y)(ba)^{-1}$
as asserted in Proposition~\ref{propositionBispecial}.
\end{example}
\section{Main result}\label{sectionMainResult}
We will prove the following result. 

\begin{theorem}\label{proposition1}
Let $A$ be three letter alphabet.
Let $F$ be a  uniformly recurrent set on the alphabet $A$
containing $A$ and
satisfying the
tree condition. 
For any $x\in F$, the set $R_F(x)$ is a basis of the free
group on $A$.
\end{theorem}
The proof uses some preliminary results.
Proposition~\ref{propositionBispecial} shows that it is enough to prove
Theorem~\ref{proposition1} for a bispecial word $x$.

Let $G$ be a graph labeled by words on the alphabet $A$ and let let
$E$
be its set of edges. We define $E^{-1}=\{(p,x^{-1},q)\mid (q,x,p)\in E\}$.
A \emph{generalized path} in $G$ is a path in the graph having
$E\cup E^{-1}$ as set of edges. Given a vertex $v$ of $G$,
 the set of labels of generalized paths
in $G$ from $v$ to $v$ is a subgroup of the free group
on $A$.  It is called the group \emph{defined} by the graph $G$
 with respect to  $v$. We will prove the following statement
(which holds without hypothesis on the number of letters).
\begin{proposition}\label{proposition3}
The group defined by any Rauzy graph of a tree set containing the
alphabet $A$
with respect to any vertex is the free group on $A$.
\end{proposition}

We will prove by checking the possible types of Rauzy graphs
the following result.
\begin{proposition}\label{propositionReturn}
Let $F$ be a  uniformly recurrent set on a three letter alphabet $A$,
containing $A$ and
satisfying the
tree condition.
For any bispecial word $x$, 
the set $R_F(x)$ has three elements and generates the 
group defined by the Rauzy graph of order $n=|x|$.
\end{proposition}
Together with Proposition~\ref{proposition3}, this implies that
$R_F(x)$ is a basis of the free group on $A$ and 
thus, by Proposition~\ref{propositionBispecial}, Theorem~\ref{proposition1}.

\subsection{Stallings foldings of Rauzy graphs}
A \emph{morphism} $\varphi$ from a labeled graph $G$ onto a labeled graph
$H$ is a map from the set of vertices of $G$ onto the set
of vertices of $H$ such that $(u,a,v)$ is an edge of $H$ if and only
if there is an edge $(p,a,q)$ of $G$ such that $\varphi(p)=u$
and $\varphi(q)=v$. An \emph{isomorphism} of labeled graphs
 is a bijective morphism.

The \emph{quotient} of a labeled graph $G$ by an equivalence $\theta$,
denoted $G/\theta$,
is the graph with vertices the set of equivalence classes of $\theta$
and an edge from the class of $u$ to the class of $v$
labeled $a$ if there is an edge labeled $a$ from a vertex $u'$
equivalent to $u$ to a vertex $v'$ equivalent to $v$.
The map from a vertex of $G$ to its equivalence class is
a morphism from $G$ onto $G/\theta$.

We consider on the graph $G_n$ the equivalence $\theta_n$
formed by the pairs
$(u,v)$ with  $u=ax$, $v=bx$, $a,b\in L(x)$
such that there is
a path from $a$ to $b$ in the graph $G(x)$ (and more precisely
from the vertex corresponding to $a$ to the vertex corresponding
to be in the copy corresponding to $L(x)$ in the bipartite graph
$G(x)$).
\begin{proposition}\label{propRauzyGraphs}
If $F$ satisfies the tree condition, for each $n\ge 1$, the quotient
of $G_n$ by the equivalence $\theta_n$ is isomorphic to $G_{n-1}$.
\end{proposition}
\begin{proof}
The map $\varphi:A^n\rightarrow A^{n-1}$ mapping a word of length $n$
to its suffix of length $n-1$ is clearly a morphism from $G_n$
onto $G_{n-1}$. If $u,v\in A^n$ are equivalent modulo $\theta_n$,
then $\varphi(u)=\varphi(v)$. Thus there is a morphism $\psi$ from
$G_n/\theta_n$ onto $G_{n-1}$. It is defined for any word $u\in F\cap
A^n$ by
$\psi(\bar{u})=\varphi(u)$
where $\bar{u}$ denotes the class of $u$ modulo $\theta_n$.
But since $F$ satisfies the tree
condition,
the class modulo $\theta_n$
of a word $ax$ of length $n$ has $\ell(x)$ elements, which is the same
as
the number of elements of $\varphi^{-1}(x)$.
This shows that $\psi$ is an isomorphism.
\end{proof}
Let $G$ be a labeled graph.
A \emph{Stallings folding} at vertex $v$ relative to letter $a$ of $G$
consists in 
identifying the edges coming into $v$ labeled $a$ and identifying
their origins. A Stallings folding does not modify the group defined
by the graph.
\begin{proofof}{of Proposition~\ref{proposition3}}
The quotient $G_n/\theta_n$ can be obtained by a sequence of Stallings
foldings from the graph $G_n$. Indeed, a Stallings folding at vertex $v$ identifies
vertices
which are equivalent modulo $\theta_n$.
Conversely, consider $u=ax$ and $v=bx$, with $a,b\in A$ such that
$a$ and $b$ (considered as elements of $L(x)$),
are connected by a path in $G(x)$. Let $a_0,\ldots a_k$ and
$b_1,\cdots b_{k}$ with $a=a_0$ and $b=a_k$ be such that
$(a_i,b_{i+1})$ for $0\le i\le k-1$
 and $(a_i,b_i)$ for $1\le i\le k$ are in $E(x)$. The successive
Stallings foldings at $xb_1,xb_2,\ldots,xb_k$ identify the
vertices $u=a_0x,a_1x,\ldots,a_kx=v$.
Indeed, since $a_ixb_{i+1}, a_{i+1}xb_{i+1} \in F$,
there are two edges labeled  $b_{i+1}$
going out of $a_ix$ and $a_{i+1}x$ which end at
$xb_{i+1}$. The Stallings folding identifies
 $a_ix$ and $a_{i+1}x$. The conclusion follows by induction.

Since the Stallings folding do not modify the group recognized, we
deduce
from Proposition~\ref{propRauzyGraphs} that the group defined
by the Rauzy graph $G_n$ is the same as the group defined by
the Rauzy graph $G_0$. Since $G_0$ is the graph with
one vertex and with loops labeled by each of the letters,
it defines the free group on $A$.
\end{proofof}
\begin{example}
Let $y$ be the infinite word obtained by decoding the Fibonacci word
into blocks of length $2$. Set $u=aa$, $v=ab$, $w=ba$. The graph
$G_2$ is represented on the left of Figure~\ref{figFiboBlocks}.
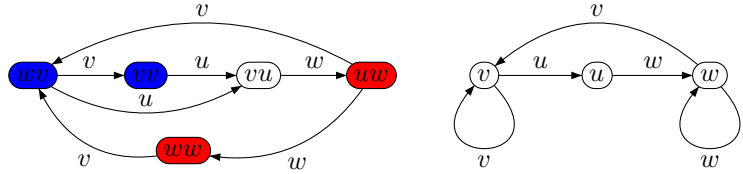
\begin{figure}[hbt]
\centering
\gasset{Nadjust=wh}
\begin{picture}(100,20)
\node[fillcolor=blue](wv)(0,10){$wv$}\node[fillcolor=blue](vv)(15,10){$vv$}\node(vu)(30,10){$vu$}\node[fillcolor=red](uw)(45,10){$uw$}
\node[fillcolor=red](ww)(20,0){$ww$}
\drawedge(wv,vv){$v$}\drawedge(vv,vu){$u$}\drawedge(vu,uw){$w$}
\drawedge[curvedepth=-7,ELside=r](uw,wv){$v$}\drawedge[curvedepth=5](uw,ww){$w$}
\drawedge[curvedepth=5](ww,wv){$v$}
\drawedge[curvedepth=-5](wv,vu){$u$}

\node(v)(60,10){$v$}\node(u)(75,10){$u$}\node(w)(90,10){$w$}
\drawloop[loopangle=-90](v){$v$}\drawedge(v,u){$u$}\drawedge(u,w){$w$}
\drawloop[loopangle=-90](w){$w$}\drawedge[curvedepth=-7,ELside=r](w,v){$v$}
\end{picture}
\caption{The Rauzy graphs $G_2$ and $G_1$ for the decoding of the Fibonacci word
  into blocks of length $2$.}\label{figFiboBlocks}
\end{figure}
The classes of $\theta_2$ are indicated with colors. The graph $G_1$
is represented on the the right.
\end{example}
The following example shows that Proposition~\ref{propRauzyGraphs} is false for sets
which do not satisfy the tree condition.
\begin{example}
Let $A=\{a,b,c\}$.
The \emph{Chacon word} on three letters
is the fixpoint $x=f^\omega(a)$ of the morphism $f$ from
$A^*$ into itself defined by $f(a)=aabc$, $f(b)=bc$ and $f(c)=abc$.
Thus $x=aabcaabcbcabc\cdots$. The \emph{Chacon set} is the set $F$ of
factors of $x$. It is of complexity $2n+1$ (see~\cite{PytheasFogg2002}
Section 5.5.2).
The Rauzy graph $G_1$ corresponding to the Chacon set is
represented in Figure~\ref{figChacon2} on the left.
The graph $G_1/\theta_1$ is represented on the right. It is
not isomorphic to $G_0$ since it has two vertices instead of one.
\begin{figure}[hbt]
\centering
\gasset{Nadjust=wh}
\begin{picture}(80,20)
\node(a)(0,5){$a$}\node(b)(20,5){$b$}\node(c)(40,5){$c$}

\drawloop[loopangle=180](a){$a$}\drawedge(a,b){$b$}
\drawedge(b,c){$c$}\drawedge[curvedepth=-5,ELside=r](c,a){$a$}\drawedge[curvedepth=5](c,b){$b$}

\node(ac)(60,5){}\node(b')(80,5){$b$}
\drawloop[loopangle=180](ac){$a$}\drawedge[curvedepth=5](ac,b'){$b$}
\drawedge[curvedepth=5](b',ac){$c$}
\end{picture}
\caption{The graphs $G_1$ and $G_1/\theta_1$.}\label{figChacon2}
\end{figure}
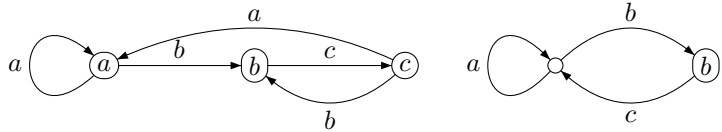
\end{example}

\subsection{Types of Rauzy graphs}
Let $G$ be a graph labeled by words on an alphabet $A$. 
Call left-special
(resp. right-special) a vertex of $G$ which has more than one predecessor
(resp. successor). Similarly a vertex is bispecial if it
is left and right-special.
The
\emph{type} of $G$ is the graph $G'$ defined as follows. Its vertices 
are the left-special and
right-special
vertices of $G$. There is an edge from $p$ to $q$ in $G'$ if there
is a path $\pi$ from $p$ to $q$ in $G$ which uses between $p$ and $q$ only
vertices which are not vertices of $G'$. The label of this edge
is the label of $\pi$.

The possible types of Rauzy graphs for a set of words
on a three letter alphabet with complexity $2n+1$ 
(and thus in particular for tree sets) have
been described in~\cite{Santini-Bouchard1997}. 
There are $9$ types of Rauzy graphs with a bispecial vertex $x$.
Seven of them are such that all cycles use this vertex.
This implies that the set of paths of first returns to  $x$
is finite and, for this reason we refer to this case as the
\emph{finite case}. The $7$ graphs are represented in Figure~\ref{figFinite}.

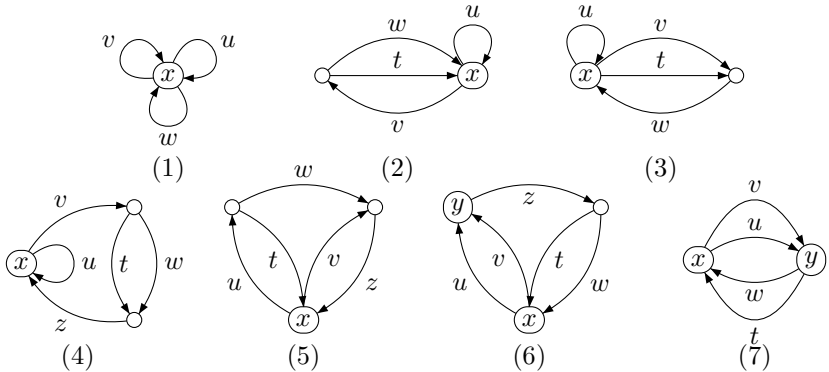
\begin{figure}[hbt]
\centering
\gasset{Nadjust=wh,loopdiam=5}
\begin{picture}(120,50)
\put(12,30){
\begin{picture}(15,15)
\node(m)(7.5,7.5){$x$}
\drawloop[loopangle=30](m){$u$}
\drawloop[loopangle=150](m){$v$}
\drawloop[loopangle=-90](m){$w$}
\node[Nframe=n](1)(7.5,-5){$(1)$}
\end{picture}
}
\put(40,30){
\begin{picture}(20,20)
\node(gm)(0,7.5){}
\node(dm)(20,7.5){$x$}

\drawloop[loopdiam=5](dm){$u$}\drawedge[curvedepth=5](gm,dm){$w$}
\drawedge[curvedepth=5](dm,gm){$v$}
\drawedge(gm,dm){$t$}
\node[Nframe=n](2)(10,-5){$(2)$}
\end{picture}
}
\put(75,30){
\begin{picture}(20,20)
\node(dm)(20,7.5){}
\node(gm)(0,7.5){$x$}

\drawloop[loopangle=90,loopdiam=5](gm){$u$}
\drawedge[curvedepth=5](gm,dm){$v$}\drawedge[curvedepth=5](dm,gm){$w$}
\drawedge(gm,dm){$t$}
\node[Nframe=n](3)(10,-5){$(3)$}
\end{picture}
}
\put(0,5){
\begin{picture}(15,15)
\node(dh)(15,15){}
\node(gm)(0,7.5){$x$}\node(db)(15,0){}

\drawloop[loopangle=0](gm){$u$}
\drawedge[curvedepth=3](gm,dh){$v$}\drawedge[curvedepth=3](db,gm){$z$}
\drawedge[curvedepth=3](dh,db){$w$}\drawedge[curvedepth=-3](dh,db){$t$}
\node[Nframe=n](4)(7.5,-5){$(4)$}
\end{picture}

}
\put(30,5){
\begin{picture}(15,15)
\node(gh)(-2,15){}\node(dh)(17,15){}
\node(mb)(7.5,0){$x$}

\drawedge[curvedepth=3](gh,dh){$w$}
\drawedge[curvedepth=3,ELside=r](gh,mb){$t$}\drawedge[curvedepth=3](mb,gh){$u$}
\drawedge[curvedepth=3](dh,mb){$z$}\drawedge[curvedepth=3,ELside=r](mb,dh){$v$}
\node[Nframe=n](5)(7.5,-5){$(5)$}
\end{picture}
}
\put(60,5){
\begin{picture}(15,15)
\node(gh)(-2,15){$y$}\node(dh)(17,15){}
\node(mb)(7.5,0){$x$}
\drawedge[curvedepth=3,ELside=r](gh,dh){$z$}
\drawedge[curvedepth=3](mb,gh){$u$}\drawedge[curvedepth=-3](mb,gh){$v$}
\drawedge[curvedepth=3](dh,mb){$w$}\drawedge[curvedepth=-3](dh,mb){$t$}
\node[Nframe=n](6)(7.5,-5){$(6)$}
\end{picture}

}
\put(90,5){
\begin{picture}(15,15)
\node(dm)(15,8){$y$}
\node(gm)(0,8){$x$}
\drawedge[curvedepth=3](gm,dm){$u$}\drawedge[curvedepth=8](gm,dm){$v$}
\drawedge[curvedepth=3](dm,gm){$w$}\drawedge[curvedepth=8](dm,gm){$t$}
\node[Nframe=n](7)(7.5,-5){$(7)$}
\end{picture}
}

\end{picture}
\caption{The finite case. }\label{figFinite}
\end{figure}
The two remaining types of graphs are such that the
set of paths of first returns to the bispecial vertex is infinite.
We refer to this case as the \emph{infinite case}.
The two types of graphs are represented in Figure~\ref{figInfinite}.

\subsection{The finite case}

We consider in turn the seven graphs.

\paragraph{Cases 1 to 5.} 
 In the first five cases, the
proof is straightforward because for each path of first return to the
bispecial vertex $x$, there is an edge which is only on this path.
 Since
each edge is on an infinite path with its label in $F$ , this implies that
the labels of all paths of first
return
to the bispecial vertex $x$ are in $F$. Thus
Proposition~\ref{propositionReturn} 
holds in these cases. We list in each of them the set $R(x)$.
\begin{displaymath}
\begin{array}{|c|c|}\hline
\text{Case}&R(x)\\ \hline
1&\{u,v,w\}\\ \hline
2&\{u,vw,vt\}\\ \hline
3&\{u,vw,tw\}\\ \hline
4&\{u,vwz,vtz\}\\ \hline
5&\{ut,uwz,vz\}\\ \hline
\end{array}
\end{displaymath}
\paragraph{Case 6.} Set $Z=\{uzw,uzt,vzt,vzw\}$. We define
\begin{displaymath}
U=\{xu,xv\}y^{-1},\quad V=\{w,t\}.
\end{displaymath}
In view of  applying Proposition~\ref{PropStrongTreeCondition} to the
generalized
extension graph
$G_{U,V}(yz)$, we prove the following.
\begin{proposition}
The set $U$ is an $F(yz)^{-1}$-maximal suffix code.
\end{proposition}
\begin{proof}
Each word $xu,xv$ contain only one occurrence of $x$. Thus $\{xu,xv\}$
is a suffix code and consequently also $\{xu,xv\}y^{-1}$. Let
$s$ be such that $syz\in F$. Since the vertex $y$ can only be reached
from $x$ by a path labeled $u$ or $v$, the word $sy$
 is comparable for the suffix order with
a word in $\{xu,xv\}$. Consequently, $s$ is comparable for the suffix
order with a word in $U$. This shows that $U$ is an $F(yz)^{-1}$-maximal suffix code.
\end{proof}
\begin{proposition}
The set $V$ is a $(yz)^{-1}F$-maximal prefix code.
\end{proposition}
\begin{proof}
Since the words $w,t$ begin with distinct letters, the set $V$
is a prefix code. Let $p$ be such that $yzp\in F$. Then $p$
is prefix comparable with $t$ or $w$. This implies the conclusion.
\end{proof}

We can now apply Proposition~\ref{PropStrongTreeCondition} to the graph 
$G_{U,V}(yz)$. Since this graph is a tree, it has three edges
and this implies that $R(x)$ has three elements and generates the
same group as $Z$. Thus Proposition~\ref{propositionReturn} holds.

\paragraph{Case 7.} Set $Z=\{uw,ut,vw,vt\}$. The proof is the same
as previously using $U=\{xu,xv\}y^{-1}$, $V=\{w,t\}$ and considering
the
graph $G_{U,V}(y)$.
 
\subsection{The infinite case}
We consider in turn the two cases.
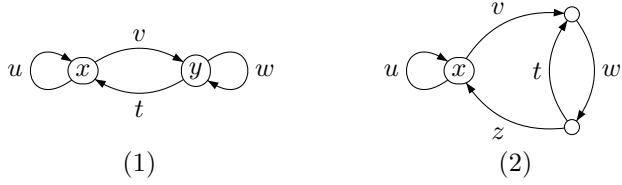
\begin{figure}[hbt]
\centering
\gasset{Nadjust=wh}
\begin{picture}(80,25)
\put(0,5){
\begin{picture}(15,15)
\node(dm)(15,7.5){$y$}
\node(gm)(0,7.5){$x$}
\drawloop[loopangle=180,loopdiam=5](gm){$u$}\drawloop[loopangle=0,loopdiam=5](dm){$w$}
\drawedge[curvedepth=3](gm,dm){$v$}\drawedge[curvedepth=3](dm,gm){$t$}
\node[Nframe=n](1)(7.5,-5){$(1)$}
\end{picture}
}
\put(50,5){
\begin{picture}(15,15)
\node(dh)(15,15){}
\node(gm)(0,7.5){$x$}\node(db)(15,0){}

\drawloop[loopangle=180,loopdiam=5](gm){$u$}
\drawedge[curvedepth=3](gm,dh){$v$}\drawedge[curvedepth=3](db,gm){$z$}
\drawedge[curvedepth=3](dh,db){$w$}\drawedge[curvedepth=3](db,dh){$t$}
\node[Nframe=n](2)(7.5,-5){$(2)$}
\end{picture}
}
\put(60,0){

}

\end{picture}
\caption{The infinite case}\label{figInfinite}
\end{figure}

\paragraph{Case 1.} Set $X=xvw^*$. In view of applying the strong tree
condition to a graph of the form $G_{U,V}(yw^n)$, we prove the following.
\begin{proposition}\label{propositionU}
Let $n\ge 0$ be such that $yw^n\in F$. Set $z=yw^n$ and $U=Xy^{-1}\cap
Fz^{-1}$.
The set $U$ is an $Fz^{-1}$-maximal suffix code.
\end{proposition}
\begin{proof}

A word in $X$ has only one occurrence of $x$ which its
prefix
of length $n$. This implies that $X$ is a suffix code. 
Since any word in $X$ has $y$ as a suffix, the set $Xy^{-1}$
is also a suffix code and therefore also $U=Xy^{-1}\cap
Fz^{-1}$.

Next, let
$s$ be such that $sz\in F$. Then $sy$ is suffix comparable with
a word beginning with $x$ and thus with
a word in $xvw^*$. Thus $s$ is suffix comparable with a word in $Xy^{-1}$.
\end{proof}

Set $Y=w^*t$. Symmetrically, we have
\begin{proposition}\label{propositionV}
Let $n\ge 0$ be such that $yw^n\in F$. Set $z=yw^n$ and $V=Y\cap
z^{-1}F$.
The set $V$ is a $z^{-1}F$-maximal prefix code.
\end{proposition}
\begin{proof}
Since $w,t$ have distinct initial letters, the set $\{w,t\}$ is a
prefix code and consequently $Y$ is a prefix code. Thus $V$
is a prefix code.

If $zp\in F$, then $zp$ is prefix comparable with a word ending in $x$
and
thus with a word in $Y$.
\end{proof}

\begin{proposition}\label{proposition2}
 We have $R_F(x)=\{u,vw^nt,vw^{n+1}t\}$
for some $n\ge 0$.
\end{proposition}
\begin{proof}
Since $F$ is uniformly recurrent, the set $R_F(x)$ is finite.
Moreover, there is an $n\ge 1$ such that $vw^nt\in R_F(x)$. Indeed
otherwise
the word $w$ would not be a factor of $\Gamma(x)$.

Let us show that the set $R_F(x)\cap vw^*t$ cannot be reduced to one
element. 
Indeed, assume that $R_F(x)\cap vw^*t$ is reduced to the word $vw^nt$.
As we have just seen, we have $n\ge 1$.
Consider the graph $G_{U,V}(y)$ with $U=Xy^{-1}\cap Fy^{-1}$ and $V=Y\cap
y^{-1}F$. By
Proposition~\ref{propositionU}, the set $U$ is an
$Fy^{-1}$-maximal
suffix code. By Proposition~\ref{propositionV}, $V$ is an
$y^{-1}F$-maximal prefix code. Since $F$ is a tree set, by
Proposition~\ref{PropStrongTreeCondition}, the
generalized extension
graph graph $G=G_{U,V}(y)$ is a
tree. 
\begin{figure}[hbt]
\centering
\gasset{Nadjust=wh,AHnb=0}
\begin{picture}(30,20)
\node(xv)(0,20){$xvy^{-1}$}\node(t)(40,20){$t$}
\node[Nframe=n](g)(0,10){$\vdots$}\node[Nframe=n](d)(40,10){$\vdots$}
\node(xvw)(0,0){$xvw^{n}y^{-1}$}\node(wt)(40,0){$w^{n}t$}

\drawedge(xv,wt){}\drawedge(xvw,t){}
\end{picture}
\caption{The graph $G_{U,V}(y)$ when $R_F(x)=\{vw^nt\}$.}\label{figureGraph}
\end{figure}
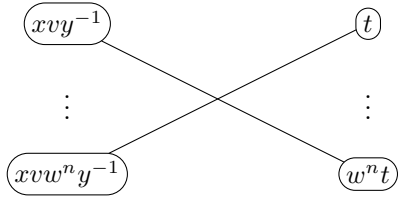
But, since $\{xvy^{-1},\ldots,xvw^{n}y^{-1}\}\subset U$ and
$\{t,\ldots,w^{n}t\}\subset V$,
and since $n\ge 1$,  the graph $G$ has at least four vertices (see
Figure~\ref{figureGraph}).  Since it has only two edges,
 it is not connected, a contradiction.

We have thus proved that $R_F(x)\cap vw^*t$ has at least two
elements. Set  $E=\{i\ge 0\mid vw^it\in
R_F(x)\}$. Let
$m,n$ with $m<n$ be the two largest elements of the set $E$.
We show that $n=m+1$. Set $k=n-m$. Assume by contradiction that $k\ge 2$.
Set  $z=yw^m$ and consider the graph $G=G_{U,V}(z)$ with $U=Xy^{-1}\cap Fz^{-1}$
and $V=Y\cap z^{-1}F$.
\begin{figure}[hbt]
\centering
\gasset{Nadjust=wh,AHnb=0}
\begin{picture}(30,20)
\node(xv)(0,20){$xvy^{-1}$}\node(t)(40,20){$t$}
\node[Nframe=n](g)(0,10){$\vdots$}\node[Nframe=n](d)(40,10){$\vdots$}
\node(xvw)(0,0){$xvw^ky^{-1}$}\node(wt)(40,0){$w^kt$}

\drawedge(xv,t){}
\drawedge[ELpos=30](xvw,t){}\drawedge[ELpos=70](xv,wt){}
\end{picture}
\caption{The graph $G_{U,V}(z)$ when $vw^nt,vw^mt\in R_F(x)$.}\label{figureGraph2}
\end{figure}
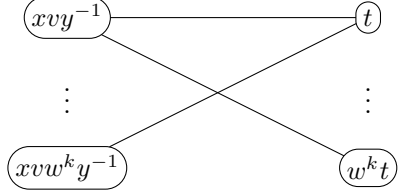
By Proposition~\ref{propositionU}, the set $U$ is an $Fz^{-1}$-maximal
suffix code and by Proposition~\ref{propositionV}, the set $V$ is a
$z^{-1}F$-maximal prefix code. Thus $G$ is a tree.
We have $\{xvy^{-1},\ldots,xvw^{k}y^{-1}\}\subset U$ and
$\{t,\ldots,w^kt\}\subset V$.
The vertices in the set $\{xvy^{-1},xvw^{k}y^{-1}\}\subset U$ are only connected to
vertices in the set $\{t,w^kt\}\subset V$
(see Figure~\ref{figureGraph2}). Since $k\ge 2$, this implies that
 $G$ is not connected, a contradiction.

We thus have proved that the two maximal elements of the set $E$ are
$n,n+1$. We finally show that $E=\{n,n+1\}$.
Consider indeed the graph $G=G_{U,V}(y)$ with $U=Xy^{-1}\cap Fy^{-1}$
and
$V=Y\cap y^{-1}F$. Any vertex $u_i=xvw^iy^{-1}\in U$ with $0<i\le n$ is connected to its
two neighbours $u_{i-1}$ and $u_{i+1}$ by a path of length $2$. Indeed, there are edges
from $u_{i-1}$ and $u_i$ to $w^{n-i+1}t$ and edges  from $u_i$ and $u_{i+1}$ to
$v_{n-i}w^{n-i}t$ (see Figure~\ref{figureGraph3}). Similarly, there is an edge from any vertex $v_i=w^it\in V$
with $0<i\le n$ to its two neighbours.
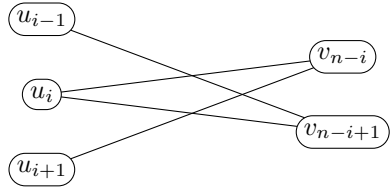
\begin{figure}[hbt]
\centering
\gasset{Nadjust=wh,AHnb=0}
\begin{picture}(50,40)
\node(ui+1)(0,0){$u_{i+1}$}\node(vn-i+1)(40,5){$v_{n-i+1}$}
\node(ui)(0,10){$u_{i}$}\node(vn-i)(40,15){$v_{n-i}$}
\node(ui-1)(0,20){$u_{i-1}$}

\drawedge(ui+1,vn-i){}\drawedge(ui,vn-i){}\drawedge(ui,vn-i+1){}
\drawedge(ui-1,vn-i+1){}
\end{picture}
\caption{The graph $G_{U,V}(y)$.}\label{figureGraph3}
\end{figure}
This shows that the graph $G$ is connected. Therefore, the set $E$
cannot
contain any other element since it would create a cycle in the graph $G$.

\end{proof}
Set  $Z=u\cup vw^*t$.
\begin{corollary}\label{corollary1}
The group generated by $R_F(x)$ contains $Z$.
\end{corollary}
\begin{proof}
 This results from Proposition~\ref{proposition2} and the fact that for all $i\ge 1$
\begin{displaymath}
vw^{i-1}t=vw^{i}t(vw^{i+1}t)^{-1}vw^{i}t,\quad vw^{i+2}t=vw^{i+1}t(vw^{i}t)^{-1}vw^{i+1}t
\end{displaymath}
\end{proof}
This implies that Proposition~\ref{propositionReturn} is true in this case.

\paragraph{Case 2.} The second case can be reduced to the previous one
considering the equivalent graph represented in
Figure~\ref{graphCase2}.
This graph is equivalent in the sense that the set of labels of paths
of first return to $x$ is the same.
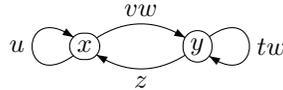
\begin{figure}[hbt]
\centering\gasset{Nadjust=wh}

\begin{picture}(15,15)
\node(dm)(15,7.5){$y$}
\node(gm)(0,7.5){$x$}
\drawloop[loopangle=180,loopdiam=5](gm){$u$}\drawloop[loopangle=0,loopdiam=5](dm){$tw$}
\drawedge[curvedepth=3](gm,dm){$vw$}\drawedge[curvedepth=3](dm,gm){$z$}

\end{picture}
\caption{An equivalent graph for case 2.}\label{graphCase2}
\end{figure}
\subsection{Examples}

Set $u=aa,v=ab,w=ba$ and $B=\{u,v,w\}$
Let $x$ be the infinite word on the alphabet $B$ obtained by decoding the Fibonacci words
in blocks of length $2$.
Let $F$ be the set of factors of $x$. There are bispecial words of
length $1$, $2$, $3$ and $5$ (but none of length $4$).
The graphs $G_1$ and $G_5$ illustrate the infinite case while $G_2$
and $G_3$ illustrate the finite case.

The Rauzy graphs $G_1$ and $G_5$
are
 represented in Figure~\ref{figureExample}. They correspond both to
 the infinite case (case 1).
\begin{figure}[hbt]
\centering
\gasset{Nadjust=wh}
\begin{picture}(100,15)(0,-5)
\put(0,0){
\begin{picture}(50,15)
\node(v)(0,0){$v$}\node(w)(20,0){$w$}

\drawloop[loopangle=180](v){$v$}\drawloop[loopangle=0](w){$w$}
\drawedge[curvedepth=5](v,w){$uw$}\drawedge[curvedepth=5](w,v){$v$}
\end{picture}
}
\put(60,0){
\begin{picture}(50,15)
\node(x)(0,0){$wvvuw$}\node(y)(20,0){$vuwwv$}

\drawloop[loopangle=180](x){$vvuw$}\drawloop[loopangle=0](y){$uwwv$}
\drawedge[curvedepth=5](x,y){$wv$}\drawedge[curvedepth=5](y,x){$vuw$}
\end{picture}
}
\end{picture}
\caption{The Rauzy graphs $G_1$ and $G_5$.}\label{figureExample}
\end{figure}
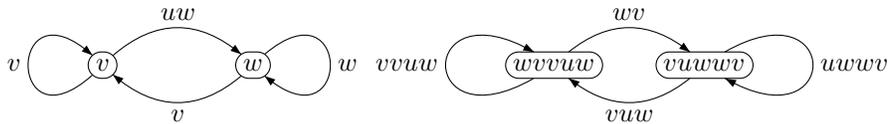
We have $R_F(v)=\{v,uwv,uwwv\}$ which corresponds to the case $n=0$ in
Proposition~\ref{proposition2}
and 
\begin{displaymath}
R(wvvuw)=\{vvuw,wv(uwwv)vuw,wv(uwwv)^2vuw\}
\end{displaymath}
 which corresponds to
the case $n=1$.

The Rauzy graphs $G_2$ and $G_3$ are represented in
Figure~\ref{figureExampleFinite}.

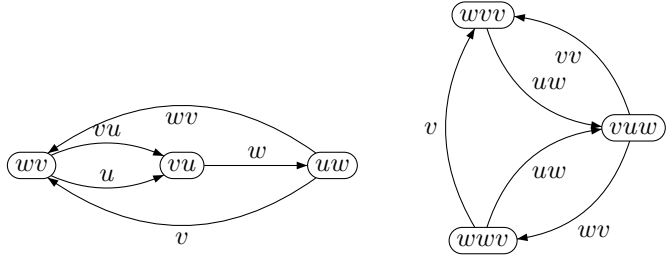
\begin{figure}[hbt]
\gasset{Nadjust=wh}\centering
\begin{picture}(100,30)
\put(0,0){
\begin{picture}(20,20)
\node(wv)(0,10){$wv$}\node(vu)(20,10){$vu$}
\node(uw)(40,10){$uw$}

\drawedge[curvedepth=3](wv,vu){$vu$}\drawedge[curvedepth=-3](wv,vu){$u$}
\drawedge(vu,uw){$w$}\drawedge[curvedepth=8](uw,wv){$v$}
\drawedge[curvedepth=-8](uw,wv){$wv$}
\end{picture}
}
\put(60,0){
\begin{picture}(20,30)
\node(vuw)(20,15){$vuw$}\node(wvv)(0,30){$wvv$}\node(wwv)(0,0){$wwv$}

\drawedge[curvedepth=-5](vuw,wvv){$vv$}\drawedge[curvedepth=5](vuw,wwv){$wv$}
\drawedge[curvedepth=5](wwv,wvv){$v$}\drawedge[curvedepth=-5](wvv,vuw){$uw$}
\drawedge[curvedepth=5,ELside=r](wwv,vuw){$uw$}
\end{picture}
}
\end{picture}
\caption{The graphs $G_2$ and $G_3$}\label{figureExampleFinite}
\end{figure}
The word $wv$ is bispecial.
It corresponds to Case $6$ of the finite case. One has 
$R(wv)=\{uwwv,vuwv,vuwwv\}$.
The word $vuw$ is bispecial and this time it is Case $5$ 
of the finite case which is represented.
\bibliographystyle{plain}
\bibliography{note}
\end{document}